\documentclass[12pt,letterpaper]{article}
\usepackage{amsmath}
\usepackage{enumerate}
\usepackage{natbib}
\newcommand{\blind}{0}

\usepackage{makeidx}
\usepackage{amsmath,amssymb,amsthm}
\usepackage{graphicx}
\usepackage{color}
\usepackage{verbatim}
\usepackage{longtable}
\usepackage{lscape}
\usepackage{bm}
\usepackage[width=\textwidth]{caption}

\usepackage{enumitem}
\usepackage{ragged2e}
\usepackage{setspace}
\usepackage{listings}
\usepackage{algorithm} 
\usepackage{algpseudocode} 
\usepackage{amsmath,amssymb,amsfonts}
\usepackage{url}
\usepackage{listings}
\usepackage{xcolor}

\definecolor{codegreen}{rgb}{0,0.6,0}
\definecolor{codegray}{rgb}{0.5,0.5,0.5}
\definecolor{codepurple}{rgb}{0.58,0,0.82}
\definecolor{backcolour}{rgb}{0.95,0.95,0.92}

\lstdefinestyle{mystyle}{
    backgroundcolor=\color{backcolour},   
    commentstyle=\color{codegreen},
    keywordstyle=\color{magenta},
    numberstyle=\tiny\color{codegray},
    stringstyle=\color{codepurple},
    basicstyle=\ttfamily\footnotesize,
    breakatwhitespace=false,         
    breaklines=true,                 
    captionpos=b,                    
    keepspaces=true,                 
    numbers=left,                    
    numbersep=5pt,                  
    showspaces=false,                
    showstringspaces=false,
    showtabs=false,                  
    tabsize=2
}

\algnewcommand\algorithmicswitch{\textbf{switch}}
\algnewcommand\algorithmiccase{\textbf{case}}
\algnewcommand\algorithmicassert{\texttt{assert}}
\algnewcommand\Assert[1]{\State \algorithmicassert(#1)}%
\algdef{SE}[SWITCH]{Switch}{EndSwitch}[1]{\algorithmicswitch\ #1\ \algorithmicdo}{\algorithmicend\ \algorithmicswitch}%
\algdef{SE}[CASE]{Case}{EndCase}[1]{\algorithmiccase\ #1}{\algorithmicend\ \algorithmiccase}%
\algtext*{EndSwitch}%
\algtext*{EndCase}%

\definecolor{blue}{rgb}{0,0,0.9}
\definecolor{red}{rgb}{0.9,0,0}
\definecolor{green}{rgb}{0,0.9,0}
\definecolor{gray}{gray}{0.6}

\newtheorem{proposition}{Proposition}

\newtheorem{theorem}{Theorem}

\def\lam{\lambda} \def\alp{\alpha}
\def\sig{\sigma}
\def\inprod#1#2{\langle#1,#2\rangle}

\newcommand{\R}{\mathbb R}

\newcommand{\primobj}{\text{obj}_\text{primal}}
\newcommand{\dualobj}{\text{obj}_\text{dual}}

\def\diag#1{{\rm diag}(#1)}
\def\norm#1{\|#1\|}

\def\cT{{\cal T}}
\def\bfone{{\bf 1}}

\begin{document}

\def\spacingset#1{\renewcommand{\baselinestretch}%
{#1}\small\normalsize} \spacingset{1.45}


\if0\blind
{
  \title{\bf A Foray into Parallel Optimisation Algorithms for High Dimension Low Sample Space (HDLSS) Generalized Distance Weighted Discrimination problems}
	\author{
		Srivathsan Amruth\thanks{Department of Mathematics, National
			University of Singapore,
			10 Lower Kent Ridge Road, Singapore
			119076.}, \;
		Advisor: Xin Yee Lam\thanks{Department of Mathematics, National
			University of Singapore,
			10 Lower Kent Ridge Road, Singapore
			119076.}, \;
	}
  \maketitle
} \fi

\if1\blind
{
  \begin{center}
    {\large \bf UROPS report for DWD on HDLSS data}
\end{center}
} \fi
\begin{abstract}
\begin{scriptsize}
In many modern data sets, High dimension low sample size (HDLSS) data is prevalent in many fields of studies. There has been an increased focus recently on using machine learning and statistical methods to mine valuable information out of these data sets. Thus, there has been an increased interest in efficient learning in high dimensions. Naturally, as the dimension of the input data increases, the learning task will become more difficult, due to increasing computational and statistical complexities. This makes it crucial to overcome the curse of dimensionality in a given dataset, within a reasonable time frame, in a bid to obtain the insights required to keep a competitive edge. To solve HDLSS problems, classical methods such as support vector machines can be utilised to alleviate data piling at the margin. However, when we question geometric domains and their assumptions on input data, we are naturally lead to convex optimisation problems and this gives rise to the development of solutions like distance weighted discrimination (DWD), which can be modelled as a second-order cone programming problem and solved by interior-point methods when sample size and feature dimensions of the data is moderate. In this paper, our focus is on designing an even more scalable and robust algorithm for solving large-scale generalized DWD problems.
\end{scriptsize}
\end{abstract}

\section{Introduction}

First we will begin by contextualising what is meant by dimension, $d$ and sample size, $n$, in a data set and how we approach solving this problem. Consider the following problem:

$\makebox[\linewidth]{\emph{You are given a few drinks and have to determine which are wine or a beer.}}
$

A few suggested ways to determine the difference would maybe be that of the colour, the taste, the smell etc. These suggested ways are alternatively known as features. If we mapping each of these features to an axis on a graph we can then generate a hypercube of possibilities. Following which if we use continuous and discrete mappings of these axises within the hypercube to place each drink, alternatively known as each sample, we can generate a point cloud of information.

To then make sense of this information we can utilise many techniques. A classically famous example is that of Support Vector Machine(SVM) where we use a hyperplane to slice the point cloud of information based on clustering information to draw a "line" between where beer exists and wine exists in our information space (hypercube).

Extending this idea is main goal of this paper where we tackle the specific problem of a low sample size (sparse point cloud) but high dimension (many axis) data. This is in line with the "Curse of dimensionality problem" where increasing dimensionality results in rapid expansion of volume encapsulated by the hypercube implying even increased data sparseness. 

Quantifying this, our paper aims to specifically solve a problem with sample size $n \approx 10^4$--$10^6$ and/or the dimension $d\approx 10^4$--$10^5$. This is done by extending the Distance Weighted Discrimination approach and aiming to solve large scale DWD problems by designing new methods with the inspiration from existing DWD implementation algorithms.

To understand how this is possible, we first need to understand the history of ADMM where classical ADMM was initially proposed for solving a 2-block convex optimization problem with a collection of coupling linear constraints. Following this, there has been a rise of new and great variety of optimization problems over time. Naturally, to solve these new problems, there have been many variations of ADMM postulated and created. 

\pagebreak

One intuitive transformation that we suppose most researchers would come up with is to extend the two-block to multi-block settings. However, in \cite{directADMM}, it was shown that the directly extended ADMM may not be convergent. This is obviously then an inherent problem as convergence is key to solving any numerical optimisation problem. Thus, it is necessary to make some modifications when directly extended ADMM in order to
get  a convergent algorithm.

Thanks to the recent advances in convergent multi-block ADMM-type methods (\citealt{STY,LST,CST}) for solving convex composite quadratic conic programming problems, we generated a novel convergent 3-block semi-proximal alternating direction method of multipliers(ADMM), which is a extension of the inexact sGS-ADMM algorithm designed in \cite{CST} to solve the DWD model.

The first contribution we make is in reformulating the primal formulation of the {\em generalized} DWD model (using the
terminology from \cite{WangZou}) and adapting the
powerful inexact sGS-ADMM framework for solving the  reformulated problem.This is in contrast to numerous SVM algorithms which are primarily designed for solving
the dual formulation of the SVM model.

The second contribution we make is in designing highly efficient techniques to
solve the subproblems in each  of the inexact sGS-ADMM iterations.
If $n$ or $d$ is moderate, then the complexity at each iteration is $O(nd)+O(n^2)$ or $O(nd)+O(d^2)$ respectively. If both $n$ and $d$ are large, then we employ the  conjugate gradient iterative method for solving the large linear systems of equations involved.
We also devise various strategies to speed up the practical performance of
the sGS-ADMM algorithm in solving large scale instances (with the largest instance
having $n=256,000$ and $d\approx 3\times 10^6$) of DWD problems
with real data sets from the UCI machine learning repository (\citealt{UCI}).
We should emphasize that the key in achieving high efficiency in our algorithm depends very much on the intricate numerical techniques and
sophisticated implementation we have developed.

\pagebreak 
\section{Generalized distance weighted discrimination}\label{sec-GeneralDWD}

This section gives details on the optimization problems underlying the distance weighted discrimination.

\begin{description}[font=\sffamily\bfseries, leftmargin=1cm, style=nextline]
  \item[Training Data]
    $(x_i,y_i)$, \emph{\textbf{\begin{scriptsize}where\end{scriptsize}}} 
    \begin{itemize}
        \begin{footnotesize}
            \item $i=1,2,3,\ldots,n$
            \item $x_i\in \R^d$ is the feature vector
        \end{footnotesize}
    \end{itemize}
  \item[Corresponding Class labels]
    $y_i\in\{-1,+1\}$
  \item[X Matrix]
  columns consisting of "$x_i$" s , \emph{\textbf{\begin{scriptsize}where\end{scriptsize}}} 
    \begin{itemize}
          \begin{footnotesize}
            \item$X\in \R^{d\times n}$
          \end{footnotesize}
    \end{itemize}
  \item[y vector] 
  $y=[y_1,\ldots,y_n]^T$
  \item[Hyperplane]
  $H= \{ x\in \R^d\mid w^T x + \beta =0 \}$,  \emph{\textbf{\begin{scriptsize}where\end{scriptsize}}} 
      \begin{itemize}
        \begin{footnotesize}
            \item$w\in \R^d$ is the unit normal 
            \item$|\beta|$ is its distance to the origin
          \end{footnotesize}
    \end{itemize}

\end{description}
In linear discrimination, we attempt to separate the vectors in the two classes by
using the above hyperplane. For binary classification where the label $y_i \in \{-1,+1\}$, we want
\begin{equation*}
y_i(\beta+x_i^Tw)\geq 1-\xi_i \quad \forall\; i=1,...,n, 
\end{equation*}
 
      \begin{itemize}\setlength\itemsep{0em}
        \begin{scriptsize}
            \item $w^Tz + \beta$ \emph{(signed distance between $z$ and the hyperplane $H$)} given a point $z\in \R^d$
            \item slack variable $\xi \geq 0$ (to allow the possibility that the positive and negative data points may not be separated cleanly  by the hyperplane)
          \end{scriptsize}
    \end{itemize}

Rewriting that into matrix-vector notation, we need
\begin{equation}
r \;:=\; Z^Tw+\beta y +\xi \;\geq\; \bfone, 
\emph{\textbf{\begin{scriptsize}where\end{scriptsize}}} 
\label{eq-r}
\end{equation}
      \begin{itemize}\setlength\itemsep{0em}
        \begin{scriptsize}
            \item $Z = X \diag{y}$
            \item $\bfone\in \mathbb{R}^n$ is the vector of ones.
          \end{scriptsize}
    \end{itemize}

For the SVM approach,
\begin{eqnarray}
\max \Big\{ \delta - C \inprod{\bfone}{\xi} \mid Z^T w + \beta y +\xi \geq \delta \bfone, \; \xi \geq 0, \; w^Tw\leq 1 \Big\}, 
\emph{\textbf{\begin{scriptsize}where\end{scriptsize}}} 
\label{eq-svm}
\end{eqnarray}

      \begin{itemize}\setlength\itemsep{0em}
        \begin{scriptsize}
            \item $w$ and $\beta$ are chosen by maximizing the minimum residual
            \item $C > 0$ is a tuning parameter to control the level of penalization on $\xi$.
          \end{scriptsize}
    \end{itemize}

For the DWD approach, \textsubscript{(introduced in \cite{MarronTodd-2007})}
\begin{eqnarray}
\min \Big\{   \sum_{i=1}^n \frac{1}{r_i} + C \inprod{\bfone}{\xi} \mid
r= Z^T w + \beta y +\xi, \; r > 0, \; \xi \geq 0, \; w^Tw\leq 1, \; w\in \R^d\Big\},
\emph{\textbf{\begin{scriptsize}where\end{scriptsize}}} 
\label{eq-dwd}
\end{eqnarray}

      \begin{itemize}\setlength\itemsep{0em}
        \begin{scriptsize}
            \item $w$ and $\beta$ are chosen by minimizing the sum of reciprocals of the $r_i$'s
            \item DWD optimization problem \eqref{eq-dwd} is shown to be equivalent to a second-order
cone programming problem \textsubscript{(in \cite{MarronTodd-2007})} and  hence it can be
solved by interior-point methods \textsubscript{(such as those implemented in the solver SDPT3 (\citealt{SDPT3}))}.
          \end{scriptsize}
    \end{itemize}

{\raggedleft\vfill{%
    \begin{FlushRight}
        \begin{spacing}{0.25}
            \begin{tiny}
                \emph{Detailed discussions on the connections between the DWD model \eqref{eq-dwd}  \\ and
                    the SVM model \eqref{eq-svm} can  be found in \cite{MarronTodd-2007}.}
                \end{tiny}
            \end{spacing}
        \end{FlushRight}
}\par
}
\pagebreak

Our approach to solve large scale generalized DWD problems:
\begin{eqnarray}
\min \Big\{ \Phi(r,\xi) := \sum_{i=1}^n \theta_q (r_i) +C \inprod{ e}{ \xi} 
\mid \; Z^T w + \beta y +\xi -r =0, \;\;
 \norm{w} \leq 1,\;\xi \geq 0
 \Big\},
 \emph{\textbf{\begin{scriptsize}where\end{scriptsize}}} 
\label{eq-primal}
\end{eqnarray}

      \begin{itemize}\setlength\itemsep{0em}
        \begin{scriptsize}
            \item In the penalty term for each $\xi_i$ we allow for a general exponent $q$ and a nonuniform weight $e_i > 0$  
            
            \item $e\in \R^n$ is a given positive vector s.t $\norm{e}_\infty =1$ \textsubscript{(the last condition is for the purpose
of normalization).}
            \item $\theta_q (r_i)$ is the function defined by
\begin{eqnarray*}
	\theta_q (t) = 
		\frac{1}{t^q} &\mbox{if $t>0$, } \;\mbox{and} \;\;\;
		\theta_q (t) =  \infty &\mbox{if $t\le0$. } ,\emph{\textbf{\begin{scriptsize}where\end{scriptsize}}} 
\end{eqnarray*}
                \begin{itemize}\setlength\itemsep{0em}
                        \begin{scriptsize}
                                \item $q\in \R^+$ s.t likely q$ \in\{0.5,1,2,4\}$
                    \end{scriptsize}
            \end{itemize}
          \end{scriptsize}
    \end{itemize}

{\raggedleft\vfill{%
    \begin{FlushRight}
        \begin{spacing}{0.25}
            \begin{tiny}
                \emph{By a simple change of variables and modification of the data vector $y$, \eqref{eq-primal} can \\ also include the
case where the terms in $\sum_{i=1}^n \frac{1}{r_i^q}$ are weighted non-uniformly.}
                \end{tiny}
            \end{spacing}
        \end{FlushRight}
}\par
}
\pagebreak

	Consider the Lagrangian function associated with \eqref{eq-primal}:
	\begin{eqnarray*}
		&& \hspace{-0.7cm}
		L(r,w,\beta,\xi;\alpha,\eta,\lam) =  \mbox{$\sum_{i=1}^n$} \theta_q (r_i) + C \inprod{e}{\xi} -\inprod{\alp}{Z^Tw+\beta y + \xi-r}
		+ \frac{\lam}{2}(\norm{w}^2-1)  - \inprod{\eta}{\xi}
		\\
		&=& \mbox{$\sum_{i=1}^n$}  \theta_q (r_i) +
		\inprod{r}{\alp } +  \inprod{\xi} {Ce -\alp-\eta}  -\beta\inprod{y}{\alp} - \inprod{w}{Z\alp} + \frac{\lam}{2}(\inprod{w}{w}-1), 
  \emph{\textbf{\begin{scriptsize}where\end{scriptsize}}} 
	\end{eqnarray*}
 
      \begin{itemize}\setlength\itemsep{0em}
        \begin{scriptsize}
            \item $r \in \R^n $
            \item $w\in \R^d$
            \item $\beta\in\R$
            \item $\xi\in\R^n$
            \item $\alp\in\R^n$
            \item $\lam, \eta \geq 0$.
          \end{scriptsize}
    \end{itemize}
    
	Now,
	\begin{eqnarray*}
		&& \inf_{r_i}\Big\{ \theta_q (r_i) + \alp_i r_i \Big\}
		= \left\{ \begin{array}{ll} \kappa\, \alp_i^{\frac{q}{q+1}}
			&\mbox{if $\alp_i \geq 0$}
			\\[0pt]
			-\infty &\mbox{if $\alp_i < 0$}
		\end{array} \right.
		\\[5pt]
		&& \inf_{w} \Big\{ -\inprod{Z\alp}{w} + \frac{\lam}{2}\norm{w}^2 \Big\}
		=
		\left\{ \begin{array}{ll}
			-\frac{1}{2\lam}\norm{Z\alp}^2 & \mbox{if $\lam > 0$} \\[0pt]
			0 &\mbox{if $\lam=0$, $Z\alp = 0$} \\[0pt]
			-\infty & \mbox{if $\lam=0$, $Z\alp\not=0$}
		\end{array} \right.
		\\[5pt]
		&& \inf_{\xi} \Big\{  \inprod{\xi} {Ce -\alp-\eta} \Big\}
		= \left\{\begin{array}{ll}
			0 &\mbox{if $Ce-\alp-\eta=0$} \\[0pt]
			-\infty &\mbox{otherwise}
		\end{array} \right.,
		\\[5pt]
		&& \inf_{\beta} \Big\{ -\beta\inprod{y}{\alp}  \Big\}
		= \left\{
		\begin{array}{ll}
			0 &\mbox{if $\inprod{y}{\alp}=0$} \\[0pt]
			-\infty &\mbox{otherwise}
		\end{array} \right..
	\end{eqnarray*}

 \pagebreak
 
	\noindent  Letting $F_D = \{ \alp\in\R^n \mid 0\leq \alp\leq Ce,  \inprod{y}{\alp}=0 \}$ :
 \\Hence,
	\begin{eqnarray*}
		\min_{r,w,\beta,\xi} L(r,w,\beta,\xi;\alpha,\eta,\lam)
		=\left\{\begin{array}{l}
			\kappa \sum_{i=1}^n \alp_i^{\frac{q}{q+1}}
			- \frac{1}{2\lam}\norm{Z\alp}^2 - \frac{\lam}{2}, \quad\mbox{if $\lam > 0$, $\alp\in F_D$},
			\\[0pt]
			\kappa \sum_{i=1}^n \alp_i^{\frac{q}{q+1}}, \quad \mbox{if $\lam =0$, $Z\alp=0$, $\alp\in F_D$},
			\\[0pt]
			-\infty, \quad \mbox{if $\lam =0$, $Z\alp\not=0$, $\alp\in F_D$, or $\alp\not\in F_D$.}
		\end{array} \right.
	\end{eqnarray*}
	Now for $\alp \in F_D$,  we have
	\begin{eqnarray*}
		& \max_{\lam \geq 0, \eta \geq 0} \big\{ \min_{r,w,\beta,\xi} L(r,w,\beta,\xi;\alpha,\eta,\lam)\big \}
		= \kappa \sum_{i=1}^n \alp_i^{\frac{q}{q+1}} -\norm{Z\alp}.
	\end{eqnarray*}
	From here, let  $\kappa = \frac{q+1}{q} q^{\frac{1}{q+1}}$ and we get the required dual problem of \eqref{eq-primal} as follows:

\begin{proposition}  
	\begin{eqnarray}
	- \min_{\alp} \Big\{ \Psi(\alp) := \norm{Z\alp} - \kappa\sum_{i=1}^n \alp_i^{\frac{q}{q+1}}
	\mid 0 \leq \alp \leq Ce , \inprod{y}{\alp} =0
	\Big\},
	\label{eq-dual}
	\end{eqnarray}
\end{proposition}

\bigskip
As it is trivial to show that the feasible regions of \eqref{eq-primal} and \eqref{eq-dual}
both have nonempty interiors, an optimal solutions for both problems exist and they
satisfy
the following Karush-Kuhn-Tucker(KKT) optimality conditions:
\begin{eqnarray}
\begin{array}{ll}
Z^T w + \beta y + \xi - r = 0, & \inprod{y}{\alp} = 0, \\[0pt]
r > 0,\; \alp > 0, \quad \alp \leq C e, \quad \xi \geq 0,  & \inprod{Ce - \alp}{\xi} = 0, \\[0pt]
\alp_i = \frac{q}{r_i^{q+1}}, \; i=1,\ldots,n,  &
\mbox{either} \; w = \frac{Z\alp}{\norm{Z\alp}}, \;\; \mbox{or}\;\; Z\alp = 0, \; \norm{w}^2 \leq 1.
\end{array} \label{eq-optimality}
\end{eqnarray}

\pagebreak


\noindent
Let $(r^*,\xi^*,w^*,\beta^*)$ and $\alp^*$ be an optimal solution of
\eqref{eq-primal} and \eqref{eq-dual}, respectively.

\begin{proposition} 
\end{proposition}

$\makebox[\linewidth]{
\textbf{There exists a  positive $\delta$ such that $\alp^*_i \geq \delta \;\; \forall\; i=1,\ldots,n.$ }}
$
\begin{scriptsize}
$\makebox[\linewidth]{\textit{i.e, the optimal solution, $\alp^*$, is bounded away from 0.}}
$
\end{scriptsize}

\begin{proof} For convenience, let $F_P = \{ (r,\xi,w,\beta) \mid
	Z^T w +\beta y + \xi -r = 0, \norm{w} \leq 1, \xi \geq 0\}$ be the
	feasible region of \eqref{eq-primal}. Since $(\bfone,\bfone,0,0)\in F_P$, we have that
	$$
	C e_{\min} \xi^*_i \leq
	C \inprod{e}{\xi^*} \leq \Phi(r^*,\xi^*,w^*,\beta^*) \leq \Phi(\bfone,\bfone,0,0)
	= n + C \mbox{$\sum_{i=1}^n$} e_i
	\quad \forall\; i=1,\ldots,n,
	$$
	where $e_{\min} = \min_{1\leq i \leq n} \{ e_i\}$.
	Hence we have $0\leq \xi^* \leq  \varrho \bfone$, where 
	$\varrho := \frac{n + C \sum_{i=1}^{n} e_i}{C e_{\min}}$.
	
	Next, we establish a bound for $|\beta^*|$. Suppose $\beta^* > 0$.
	Consider an index $i$ such that $y_i = -1$. Then
	$
	0 < \beta^* = Z_i^T w^* + \xi^*_i-r^*_i \leq \norm{Z_i}\norm{w^*} + \xi^*_i
	\leq K + \varrho,
	$
	where $Z_i$ denotes the $i$th column of $Z$,
	$K  = \max_{1\leq j \leq n}\{\norm{Z_j}\}$. On the other hand, if
	$\beta^* < 0$, then we consider an index $k$ such that $y_k = 1$, and
	$
	0< -\beta^* = Z_k^T w^* + \xi^*_k-r^*_k \leq K + \varrho.
	$
	To summarize, we have that $|\beta^*| \leq K + \varrho$.
	
	Now we can establish an upper bound for $r^*$. For any $i=1,\ldots,n$, we have that
	$$
	r^*_i  = Z_i^T w^* + \beta^* y_i + \xi^*_i  \leq \norm{Z_i}\norm{w^*}
	+|\beta^*| + \xi_i^* \leq 2(K+\varrho).
	$$
	From here, we get
	$
	\alp^*_i = \frac{q}{ (r_i^*)^{q+1}}  \geq  \delta := \frac{q}{ (2K+2\varrho)^{q+1}} \quad\forall\;
	i=1,\ldots,n.
	$
	This completes the proof of the proposition.
\end{proof}

\pagebreak

\section{An inexact SGS-based ADMM for large scale DWD problems}\label{sec-sGSADMM}

This section dives into the methodologies and intuitions behind our proposed solver.
\\

An infinity indicator function 
over a set $\mathcal{C}$  is defined by:
\begin{eqnarray*}
	\delta_\mathcal{C}(x):=
\begin{cases}
	0, &\quad \text{if } x \in \mathcal{C}; \\
	+\infty, &\quad \text{otherwise.}
\end{cases}
\end{eqnarray*}

Rewriting the model \eqref{eq-primal} as:
\begin{eqnarray*}
	\min \Big\{
	\sum_{i=1}^n
	\theta_q (r_i)+ C \inprod{e}{\xi} +\delta_B(w) + \delta_{\R^n_+}(\xi)
	\mid \; Z^T w +\beta y+\xi-r=0, \; \; w\in \R^d, \;  r,\xi\in \R^n
	\Big\},
 \emph{\textbf{\begin{scriptsize}where\end{scriptsize}}} 
\end{eqnarray*}

      \begin{itemize}\setlength\itemsep{0em}
        \begin{scriptsize}
            \item $B=\{w \in \R^d \mid \norm{w}\leq 1\}$. 
            \item $\delta_B(w)$ and $\delta_{\R^n_+}(\xi)$ are infinity indicator functions
          \end{scriptsize}
    \end{itemize}

\noindent
The model above can be broken down into a convex
minimization problem with three nonlinear blocks.
By introducing an auxiliary variable $u=w$, we can reformulate it as:
\begin{eqnarray}
\begin{array}{ll}
\min &\sum_{i=1}^n \theta_q (r_i)+ C \inprod{e}{\xi}+\delta_B(u) + \delta_{\R^n_+}(\xi),
\emph{\textbf{\begin{scriptsize}where\end{scriptsize}}} 
\\[8pt]
\end{array}
\label{eq-gen}
\end{eqnarray}

      \begin{itemize}\setlength\itemsep{0em}
        \begin{scriptsize}
            
            \item $\beta \in \R$
            \item w,u $\in\R^d$
            \item r, $\xi \in \R^{n}$
            \item D$\in \R^{d\times d}$ \textsubscript{is a given positive scalar multiple of the identity matrix which is introduced for the purpose of scaling the variables.}
            \item D(w - u) =0
            \item $Z^T w+\beta y+\xi-r=0$
          \end{scriptsize}
    \end{itemize}

\pagebreak

The associated Lagrangian function is given by:
\begin{eqnarray*}
	\begin{array}{rcl}
		L_{\sigma}(r,w,\beta,\xi,u;\alpha,\rho) &=&\sum_{i=1}^n \theta_q (r_i)
		+ C \inprod{e}{\xi}+\delta_B(u) + \delta_{\R^n_+}(\xi)
		+\frac{\sigma}{2}\norm{Z^T w+\beta y +\xi -r -\sigma^{-1}\alpha}^2
		\\[8pt]
		&& +\frac{\sigma}{2}\norm{D(w-u)-\sigma^{-1}\rho}^2
		-\frac{1}{2\sigma}\norm{ \alpha}^2 -\frac{1}{2 \sigma}\norm{\rho}^2,
  \emph{\textbf{\begin{scriptsize}where\end{scriptsize}}} 
	\end{array}
\end{eqnarray*}

      \begin{itemize}\setlength\itemsep{0em}
        \begin{scriptsize}
            \item given parameter $\sig > 0$
          \end{scriptsize}
    \end{itemize}

The algorithm which we will design later is based on recent progress
in algorithms for solving multi-block convex conic programming.
In particular, our algorithm is designed based on
the inexact ADMM algorithm in \cite{CST} and we made essential use of
the inexact symmetric Gauss-Seidel
decomposition theorem in \cite{LST} to solve the subproblems
arising in each iteration of the algorithm.

We can view \eqref{eq-gen} as a linearly constrained nonsmooth
convex
programming problem with three blocks of variables grouped as
$(w,\beta)$, $r$, $(u,\xi)$.
The template for our inexact
sGS based ADMM is described next.
Note that the
subproblems need not be solved exactly as long as they satisfy some prescribed accuracy.

\begin{description}
	\item[Algorithm 1.] {\bf An inexact sGS-ADMM for solving \eqref{eq-gen}.}
	\\[5pt]
	Let $\{ \varepsilon_k\}$ be a  summable sequence of nonnegative
	nonincreasing numbers.
	Given an initial iterate $(r^0,w^0,\beta^0,\xi^0,u^0)$ in the feasible region of \eqref{eq-gen}, and
	$(\alpha^0,\rho^0)$ in the dual feasible region of \eqref{eq-gen}, choose  a $d\times d$ symmetric positive semidefinite matrix $\cT$, and
	perform the
	following steps in each iteration.
	\item[Step 1a.]
	Compute
	$$
	(\bar{w}^{k+1},\bar{\beta}^{k+1})
	\approx  \mbox{argmin}_{w,\beta}\; \Big\{L_\sigma (r^{k},w,\beta,\xi^k,u^k;\alpha^k,\rho^k)
	+ \frac{\sig}{2}\norm{w-w^k}_{\cT}^2 \Big\}.
	$$

	In particular, $(\bar{w}^{k+1},\bar{\beta}^{k+1})$ is an approximate solution to the
	following $(d+1)\times (d+1)$ linear system of equations:
	\begin{eqnarray}
	\underbrace{\begin{bmatrix} ZZ^T+ D^2 +\cT & Zy\\[5pt] (Zy)^T& y^Ty
		\end{bmatrix} }_{A}
	\begin{bmatrix} w \\[5pt] \beta
	\end{bmatrix}
	=  \bar{h}^k :=
	\begin{bmatrix}
	-Z(\xi^k-r^k-\sigma^{-1}\alpha^k)+D^2 u^k + D(\sigma^{-1}\rho^k) + \cT w^k\\[5pt]
	-y^T(\xi^k-r^k-\sigma^{-1}\alpha^k)
	\end{bmatrix}. \quad
	\label{eq-linsys}
	\end{eqnarray}
	We require the residual of
	the approximate solution $(\bar{w}^{k+1},\bar{\beta}^{k+1})$
	to satisfy
	\begin{eqnarray}
	\norm{\bar{h}^k - A[\bar{w}^{k+1};\bar{\beta}^{k+1}]} \leq \varepsilon_k.
	\label{eq-linsys-tol}
	\end{eqnarray}
	
	\item[Step 1b.] Compute $r^{k+1}\approx\mbox{argmin}_{r\in \R^n}\; L_{\sigma}(r,\bar{w}^{k+1},\bar{\beta}^{k+1},\xi^k,u^k;\alpha^k,\rho^k)$. Specifically, by observing that the
	objective function in this subproblem is actually separable in $r_i$ for $i=1,\ldots,n$,
	we can compute $r^{k+1}_i$ as follows:
	\begin{eqnarray}
	\begin{array}{lll}
	r^{k+1}_i &\approx& \arg \min_{r_i} \Big\{ \theta_q (r_i) +\frac{\sigma}{2} \norm{r_i-c^k_i}^2 \Big\}\\
	&=& \arg \min_{r_i > 0} \Big\{  \frac{1}{r_i^q}
	+\frac{\sigma}{2} \norm{r_i-c^k_i}^2 \Big\} \quad \forall \;i=1,\ldots,n,
	\end{array} \label{eq-1c}
	\end{eqnarray}
	where $c^k=Z^T \bar{w}^{k+1}+y\bar{\beta}^{k+1}+\xi^k-\sigma^{-1}\alpha^k$.
	The details on how the above one-dimensional problems are solved
	will be given later.
	The solution $r_i^{k+1}$ is deemed to be sufficiently accurate if
	\begin{eqnarray*}
		\Big|-\frac{q}{ (r_i^{k+1})^{q+1}} + \sig (r^{k+1}_i- c^k_i) \Big|\leq \varepsilon_k/\sqrt{n}
		\quad \forall\; i=1,\ldots,n.
	\end{eqnarray*}
	\item[Step 1c.] Compute
	$$
	(w^{k+1},\beta^{k+1}) \approx \mbox{argmin}_{w,\beta}\; \Big\{
	L_{\sigma}(r^{k+1},w,\beta,\xi^k,u^k;\alpha^k,\rho^k) + \frac{\sig}{2}\norm{w-w^k}_\cT^2
	\Big\},$$
	which amounts to solving the
	linear system of equations \eqref{eq-linsys}
	but with $r^k$ in the right-hand side vector $\bar{h}^k$ replaced by $r^{k+1}$.
	Let $h^k$ be the new right-hand side vector. We require the approximate solution
	to satisfy the accuracy condition that
	$$
	\norm{h^k - A[w^{k+1};\beta^{k+1}]}\leq 5 \varepsilon_k.
	$$
	Observe that the accuracy requirement here is more relaxed than that
	stated in \eqref{eq-linsys-tol} of Step 1a. The reason for doing so is that
	one may hope to use the solution $(\bar{w}^{k+1},\bar{\beta}^{k+1})$
	computed in Step 1a as an approximate solution for the current subproblem.
	If   $(\bar{w}^{k+1},\bar{\beta}^{k+1})$ indeed satisfies the above
	accuracy condition, then one can simply
	set $(w^{k+1},\beta^{k+1}) = (\bar{w}^{k+1},\bar{\beta}^{k+1})$ and
	the cost of solving this new subproblem can be saved.
	
	\item[Step 2.]
	Compute $(u^{k+1},\xi^{k+1})= \mbox{argmin}_{u,\xi}\; L_{\sigma}(r^{k+1},w^{k+1},\beta^{k+1},\xi,u;\alpha^k,\rho^k)$. By observing that the objective function is
	actually separable in $u$ and $\xi$, we can compute $u^{k+1}$ and $\xi^{k+1}$ separately
	as follows:
	\begin{eqnarray*}
		u^{k+1}&=& \arg \min \left\{ \delta_B (u) +\frac{\sigma}{2}\norm{D(u-g^k)}^2 \right\} \;=\; \begin{cases} g^k & \mbox{if  $\norm{g^k} \leq 1$}
			\\[0pt]
			g^k/\norm{g^k} & \mbox{otherwise}
		\end{cases},
		\\
		\xi^{k+1}&=&\Pi _{\R_{+}^n}\Big( r^{k+1}  -Z^T w^{k+1} -y \beta^{k+1}
		+\sigma^{-1}\alpha^k -\sigma^{-1}C e\Big),
	\end{eqnarray*}
	where $g^k=w^{k+1}-\sigma^{-1}D^{-1}\rho^k$, and $\Pi_{\R^n_+}(\cdot)$ denotes
	the projection onto $\R^n_+$.
	\item[Step 3.] Compute
	\begin{eqnarray*}
		\alpha^{k+1}&=& \alpha^k -\tau \sigma (Z^T w^{k+1}+y\beta^{k+1}+\xi^{k+1}-r^{k+1}),\\
		\rho^{k+1}&=& \rho^k -\tau \sigma D (w^{k+1}-u^{k+1}),
	\end{eqnarray*}
	where $\tau \in (0,(1+\sqrt{5})/2)$ is the steplength which is typically
	chosen to be $1.618$.
\end{description}

In our implementation of Algorithm 1, we choose the summable sequence $\{ \varepsilon_k\}_{k\geq 0}$ to
be $\varepsilon_k = c/(k+1)^{1.5}$ where $c$ is a constant that is inversely
proportional to $\norm{Z}_F$.
Next we discuss the computational cost of Algorithm 1.
As we shall see later, the most computationally intensive steps in each iteration
of the above algorithm are in solving the linear systems of equations of the form
\eqref{eq-linsys} in Step 1a and 1c. The detailed analysis of their computational costs
will be presented in  subsection \ref{subsec-linsys}.
All the other steps can be done
in at most $O(n)$ or $O(d)$ arithmetic operations, together with the
computation of $Z^T w^{k+1}$, which costs $2dn$ operations if we do not
take advantage of any possible sparsity in $Z$.

\subsection{Convergence results}

We have the following convergence theorem for the inexact sGS-ADMM, established by Chen, Sun and Toh in \citet[Theorem 1]{CST}. This theorem guarantees the convergence of our algorithm to optimality, as a merit over the possibly non-convergent directly extended semi-proximal ADMM.

\begin{theorem}
	Suppose that the system \eqref{eq-optimality} has at least one solution. Let $\{(r^k,w^k,\beta^k,\xi^k,u^k;\alpha^k,\rho^k)\}$ be the sequence generated by the inexact sGS-ADMM in Algorithm 1. Then the sequence $\{(r^k,w^k,\beta^k,\xi^k,u^k)\}$ converges to an optimal solution of problem \eqref{eq-gen} and the sequence $\{(\alpha^k,\rho^k)\}$ converges to an optimal solution to the dual of problem \eqref{eq-gen}.
\end{theorem}
\begin{proof} 	
	In order to apply the convergence result in \cite{CST}, we need to express \eqref{eq-gen}
	as follows:
	\begin{eqnarray}
	\min \Big\{p(r) + f(r,w,\beta)  + q(\xi,u)  + g(\xi,u)
	\mid\;
	 A_1^* r  + A_2^* [w;\beta] + B^*[\xi; u]  = 0 \Big\}	
	\end{eqnarray}
	where $p(r) = \sum_{i=1}^n \theta_q (r_i),
		\;\; f(r,w,\beta) \equiv 0,
		\quad q(\xi,u) = \delta_B(u)  + C \inprod{e}{\xi} + \delta_{\R^n_+}(\xi),  \;\; g(\xi,u)  \equiv 0,$
	\begin{eqnarray*}
		& A_1^* = \left(\begin{array}{c} -I \\[0pt] 0 \end{array} \right),\;
		A_2^* = \left(\begin{array}{cc} Z^T & y \\[0pt] D & 0 \end{array} \right),\;
		B^* = \left(\begin{array}{cc} I & 0 \\[0pt] 0 & -D \end{array} \right).&
	\end{eqnarray*}
	Next  we need to consider the following matrices:
	\begin{eqnarray*}
		\left(\begin{array}{c}
			A_1 \\[0pt] A_2
		\end{array}\right) \Big( A_1^*,\; A_2^* \Big) + \left(\begin{array}{ccc}
		0 & 0 & 0 \\[0pt] 0 & \cT & 0 \\[0pt] 0 & 0 & 0
	\end{array}\right)
	=\left( \begin{array}{ccc}
		I &  [-Z^T, -y] \\[0pt]
		[-Z^T,-y]^T & M
	\end{array}\right), \quad BB^* = \left(\begin{array}{cc}
	I & 0 \\[0pt] 0 & D^2
\end{array}\right),
\end{eqnarray*}
where
$$
M = \left( \begin{array}{cc}
ZZ^T + D^2 + \cT & Zy \\[0pt]
(Zy)^T & y^T y
\end{array}\right) \succ 0.
$$
One can show that $M$ is positive definite by using the Schur complement lemma.
With the conditions that $M\succ 0$ and $BB^*\succ 0$, the conditions
in Proposition 4.2 of \cite{CST} are satisfied, and hence the convergence of Algorithm 1 follows by using Theorem 1 in \cite{CST}.
\end{proof}

We note here that the convergence analysis  in \cite{CST} is highly nontrivial. 
But it is motivated by the
	proof for the simpler case of an
	exact semi-proximal ADMM that is available in Appendix B of the paper by
	\cite{FPST}.
	In that paper, one can see that the convergence proof is based on the descent 
	property of 
	a certain function, while the augmented Lagrangian function itself does not have
	such a descent property.

\subsection{Numerical computation of the subproblem \eqref{eq-1c} in Step 1b}

In the presentation of Algorithm 1, we have described how the subproblem in each step can be solved
except for the subproblem \eqref{eq-1c} in Step 1b. Now we discuss how it can be solved. Observe that
for each $i$, we need to solve a one-dimensional problem of the form:
\begin{eqnarray}
\min \Big\{ \varphi(s):= \frac{1}{s^q} + \frac{\sig}{2} (s-a)^2 \mid s > 0
\Big\},
\label{eq-varphi}
\end{eqnarray}
where $a$ is given. It is easy to
see that $\varphi(\cdot)$ is a convex function and it has a unique minimizer
in the domain $(0,\infty)$.
The optimality condition for \eqref{eq-varphi} is given by
$$
s- a = \frac{q\sig^{-1}}{s^{q+1}},
$$
where the unique minimizer $s^*$
is determined by the intersection of the line $s\mapsto s-a$ and the
curve $s\mapsto \frac{q \sig^{-1}}{s^{q+1}}$ for $s> 0$.
We propose to use Newton's method to find the
minimizer, and the template is given as follows. Given an initial iterate $s^0$, perform the
following iterations:
$$
s_{k+1} = s_k - \varphi'(s_k)/\varphi''(s_k) = s_k \left( \
\frac{q(q+2)\sig^{-1} + a s^{q+1}_k}{q(q+1)\sig^{-1} +  s^{q+2}_k}\right), \quad k=0,1,\ldots
$$
Since $\varphi^{\prime\prime}(s^*)  > 0$, Newton's method
would have a local quadratic convergence rate, and we would expect
it to converge in a small number of iterations, say less than $20$,
if a good initial point $s^0$ is given. In solving the subproblem
\eqref{eq-1c} in Step 1b, we always use the previous solution $r_i^k$ as the
initial point to warm-start Newton's method.
If a good initial point is not available, one can use the bisection technique to find one. In our tests, this technique was however never used.

Observe that the computational cost for solving the subproblem \eqref{eq-1c} in Step 1b is
$O(n)$ if Newton's method converges within a fixed number of iterations (say 20)
for all $i=1,\ldots,n$. Indeed, in our experiments, the average number of
Newton iterations required to solve \eqref{eq-varphi} for each of the instances is less than $10$.

\subsection{Efficient techniques to solve the linear system 
	\eqref{eq-linsys}}
\label{subsec-linsys}

Observe that in each iteration of Algorithm 1, we need to solve a $(d+1)\times (d+1)$ linear system of
equations \eqref{eq-linsys}
with the same coefficient matrix $A$. For large scale problems where $n$ and/or $d$ are large,
this step would constitute the most expensive part of the algorithm.
In order to solve such a linear system efficiently,
we design different techniques to solve it, depending on the dimensions $n$ and $d$.
We consider the following cases.

\subsubsection{The case where $d\ll n$ and $d$ is moderate (Direct Solver Method)}\label{subsec-directSolver}

This is the most straightforward case where
we set $\cT =0$, and
we solve
\eqref{eq-linsys} by computing the Cholesky factorization of the coefficient matrix $A$.
The cost of computing $A$ is
$2 n d^2$ arithmetic operations. Assuming that $A$ is stored, then we can
compute its Cholesky factorization at the cost of $O(d^3)$ operations, which
needs only to be performed once at the very beginning of Algorithm 1.
After that, whenever we need to solve the linear system
\eqref{eq-linsys}, we compute the right-hand-side vector at
the cost of $2nd$ operations and
solve two $(d+1)\times (d+1)$
triangular systems of linear equations at the cost of $2d^2$ operations.

\subsubsection{The case where $n\ll d$ and $n$ is moderate (SMW2 Solver Method)}\label{subsec-smw}
\def\hD{\widehat{D}}

In this case, we also set $\cT=0$. But solving the large $(d+1)\times (d+1)$
system of linear equations \eqref{eq-linsys} requires more thought.
In order to avoid inverting the high dimensional matrix $A$ directly,
we make use of the Sherman-Morrison-Woodbury formula to
get $A^{-1}$ by inverting a much smaller
$(n+1)\times (n+1)$ matrix
as shown in the following proposition.

\begin{proposition} The coefficient matrix $A$ can be rewritten as follows:
	\begin{eqnarray}
	A &=& \hD + U E U^T, \quad
	U = \left[\begin{array}{cc}
	Z  & 0\\ y^T  &\norm{y}
	\end{array}\right], \quad E = \diag{I_n,-1},
	\label{eq-Anew}
	\end{eqnarray}
	where $\hD = \diag{D,\norm{y}^2}$.
	It holds that
	\begin{eqnarray}
	A^{-1} &=& \hD^{-1} - \hD^{-1} U H^{-1} U^T \hD^{-1},
	\label{eq-Ainv}
	\end{eqnarray}
	where
	\begin{eqnarray}
	H = E^{-1} + U^T \hD^{-1} U =
	\left[\begin{array}{cc}
	I_n + Z^T D^{-1} Z +  yy^T/\norm{y}^2  &  y/\norm{y}\\[3pt] y^T/\norm{y}  & 0
	\end{array}\right].
	\label{eq-H}
	\end{eqnarray}
\end{proposition}
\begin{proof} It is easy to verify that \eqref{eq-Anew} holds and we omit the details.
	To get \eqref{eq-Ainv}, we only need to apply the Sherman-Morrison-Woodbury formula in \citet[p.50]{GVbook}
	to \eqref{eq-Anew} and perform some simplifications.
\end{proof}

\medskip

Note that in making use of \eqref{eq-Ainv} to compute
$A^{-1}\bar{h}^k$, we need to find $H^{-1}$. A rather cost effective way to
do so is to express $H$ as follows and use the Sherman-Morrison-Woodbury formula to find
its inverse:
\begin{eqnarray*}
	H =J+ \bar{y}\bar{y}^T, \quad
	J = diag(I_n + Z^TD^{-1} Z,-1), \quad \bar{y} = [y/\norm{y}; \; 1].
\end{eqnarray*}
With the above expression for $H$, we have that
\begin{eqnarray*}
	H^{-1} = J^{-1} - \frac{1}{1+\bar{y}^T J^{-1}\bar{y}} (J^{-1}\bar{y} ) (J^{-1}\bar{y})^T.
\end{eqnarray*}
Thus
to solve \eqref{eq-linsys}, we first compute the $n\times n$ matrix $I_n + Z^TD^{-1} Z$ in \eqref{eq-H} at
the cost of $2dn^2$ operations. Then we compute its Cholesky factorization  at the cost of $O(n^3)$ operations.
(Observe that even though we are solving a $(d+1)\times (d+1)$ linear system of equations
for which $d\gg n$, we only need to compute the Cholesky factorization of
a much smaller $n\times n$ matrix.)
Also, we need to compute $J^{-1}\bar{y}$ at the cost of
$O(n^2)$ operations by using the previously computed Cholesky factorization.
These computations only need to be performed
once at the beginning of Algorithm 1. After that, whenever we need to
solve a linear system of the form \eqref{eq-linsys},
we can compute $\bar{h}^k$ at the cost of $2nd$ operations, and then
make use of $\eqref{eq-Ainv}$ to get $A^{-1}\bar{h}^k$ by solving two $n\times n$
triangular systems of linear equations at the cost of $2n^2$ operations, and
performing two matrix-vector multiplications involving $Z$ and $Z^T$ at a
total cost of $4nd$ operations. To summarize, given the Cholesky factorization of the first diagonal block of $H$, the cost of solving \eqref{eq-linsys} via \eqref{eq-Ainv} is $6nd + 2n^2$ operations.

\subsubsection{The case where $d$ and $n$ are both large (PSQMR Iterative Solver Method)}

The purpose of introducing the proximal term $\frac{1}{2}\norm{w-w^k}_\cT^2$
in Steps 1a  and 1c is to make the computation of the solutions of the subproblems
easier.
However, one should note that adding the proximal term typically will make the algorithm converge more slowly, and the  deterioration will become worse for
larger $\norm{\cT}$. Thus in practice, one would need to strike a balance between choosing
a symmetric positive semidefinite matrix $\cT$ to make the computation easier while not slowing down
the algorithm by too much.

In our implementation, we first attempt to
solve the subproblem in Step 1a (similarly for 1c) without adding a proximal term by setting
$\cT=0$. In particular, we solve the linear system \eqref{eq-linsys}
by using
a preconditioned symmetric quasi-minimal residual (PSQMR)  iterative solver (\citealt{PSQMR})
when both $n$ and $d$ are large.
	Basically, it is a variant of the Krylov subspace method similar to the idea in GMRES (\citealt{SaadBook}). For more details on the PSQMR algorithm, the reader is referred to the appendix.	
In each step of the PSQMR solver, the main cost is in performing the matrix-vector
multiplication with the coefficient matrix $A$, which costs
$4nd$ arithmetic operations.
As the number of steps taken by an iterative solver to solve \eqref{eq-linsys}
to the required accuracy \eqref{eq-linsys-tol} is dependent on the
conditioning of   $A$, in the event that the
solver requires more than $50$ steps to solve \eqref{eq-linsys},
we would switch to
adding a suitable non-zero proximal term $\cT$ to make the subproblem in Step 1a easier to solve. 	 	

The most common and natural choice of $\cT$ to make the subproblem in Step 1a easy to solve is
to set $\cT = \lambda_{\max}I- ZZ^T$, where
$\lam_{\max}$ denotes the largest eigenvalue of $ZZ^T$. In this case
the corresponding linear system \eqref{eq-linsys} is very easy to solve. More precisely, for the
linear system in \eqref{eq-linsys}, we can first compute $\bar{\beta}^{k+1}$ via the
Schur complement equation in a single variable followed by computing $\bar{w}^k$ as follows:
\begin{eqnarray}
\begin{array}{l}
\big(y^Ty - (Zy)^T (\lam_{\max} I + D)^{-1} (Zy) \big)\beta = \bar{h}^k_{d+1}
- (Zy)^T (\lam_{\max} I + D)^{-1} \bar{h}^k_{1:d},
\\[5pt]
\bar{w}^{k+1} = (\lam_{\max}I + D)^{-1} (\bar{h}^k_{1:d} - (Zy)\bar{\beta}^{k+1}),
\end{array}
\label{eq-schur}
\end{eqnarray}
where $\bar{h}^k_{1:d}$ denotes the vector extracted from the first $d$ components
of $\bar{h}^k$.
In our implementation, we pick a $\cT$ which is less conservative
than the above natural choice as follows.
Suppose we have computed the first $\ell$ largest eigenvalues of $ZZ^T$ such that $\lam_1\geq\ldots\geq\lam_{\ell-1} > \lam_\ell$,  and their
corresponding orthonormal set of eigenvectors, $v_1,\ldots,v_\ell$.
We pick $\cT$ to be
\begin{eqnarray}
\cT = \lam_\ell I + \mbox{$\sum_{i=1}^{\ell-1}$} (\lam_i -\lam_\ell) v_i v_i^T - ZZ^T,
\label{eq-T}
\end{eqnarray}
which can be proved to be positive semidefinite by using the spectral decomposition
of $ZZ^T$. In practice, one would typically pick $\ell$ to be a small integer, say $10$,
and compute the first $\ell$ largest eigenvalues and their corresponding eigenvectors
via variants of the Lanczos method. 
The most expensive step in each iteration of the Lanczos method is a matrix-vector multiplication, which 
requires $O(d^2)$ operations. In general, the cost of computing the first few largest eigenvalues of $ZZ^T$ 
is much cheaper than that of computing the full eigenvalue decomposition.
In {\sc Matlab}, such a computation can be
done by using the routine {\tt eigs}.
To solve \eqref{eq-linsys}, we need the inverse of $ZZ^T + D+\cT$. Fortunately,
when $D = \mu I_d$, it can easily be inverted with
$$
(ZZ^T + D + \cT)^{-1} = (\mu+\lam_\ell)^{-1} I_d
+ \mbox{$\sum_{i=1}^{\ell-1}$}
\big( (\mu+\lam_i)^{-1}-(\mu+\lam_\ell)^{-1}\big) v_iv_i^T.
$$
One  can  then compute $\bar{\beta}^k$ and $\bar{w}^k$ as
in \eqref{eq-schur} with $(\lam_{\max} I + D)^{-1}$ replaced by the above
inverse.

\pagebreak

\section{Experiments and Analysis of Results}\label{Results}
In this section, we test the performance of our inexact sGS-ADMM method on several publicly available data sets. The numerical results presented in the subsequent subsections are obtained from a computer with processor specifications: DUAL AMD EPYC\textsuperscript{TM} 7763 CPU clocked @2.45GHz and 2048GB of RAM, running on a 64-bit Red Hat Enterprise Linux\textsuperscript{\textregistered} Operating System.

\subsection{Tuning the penalty parameter}

In the DWD model \eqref{eq-gen}, we see that it is important to make a suitable choice of the penalty parameter $C$.
In \cite{MarronTodd-2007}, it has been noticed that a reasonable choice for the penalty parameter when the exponent $q=1$ is a large constant divided by the square of a typical distance between the $x_i$'s, where the typical distance, $dist$, is defined as the median of the pairwise Euclidean distances between classes. We found out that in a more general case, $C$ should be inversely proportional to $dist^{q+1}$. On the other hand, we observed that a good choice of $C$ also depends on the sample size $n$ and the dimension of features $d$. In our numerical experiments, we empirically set the value of $C$ to be
$10^{q+1}\max \big\{1,\frac{10^{q-1} \log(n)\max\{1000,d\}^{\frac{1}{3}}}{dist^{q+1}}\big\}$,
where $\log(\cdot)$ is the natural logarithm.

\subsection{Scaling of data}

A technique which is very important in implementing ADMM based methods in practice to
achieve fast convergence
is the data scaling technique.
Empirically, we have observed that
it is good to scale the matrix $Z$ in \eqref{eq-gen} so that the magnitude of all the blocks in the equality constraint would be roughly the same. Here we choose the scaling factor to be $Z_{\rm scale} = \sqrt{\norm{X}_F}$, 
where $\norm{\cdot}_F$ is the Frobenius norm. Hence the optimization model in \eqref{eq-gen} becomes:
\begin{eqnarray}
\begin{array}{ll}
\min &\sum_{i=1}^n \frac{1}{r_i^q}+ C \inprod{e}{\xi}+\delta_{\widetilde{B}}(\tilde{u}) + \delta_{\R^n_+}(\xi)
\\[8pt]
\text{s.t.} & \widetilde{Z}^T \,\tilde{w}
+\beta y+\xi-r=0,\; r > 0,\\[0pt]
&  D(\tilde{w}- \tilde{u}) =0, \;\;  \tilde{w}, \tilde{u}\in\R^d, \; r,\xi\in \R^n,
\end{array}
\label{eq-gen-scale}
\end{eqnarray}
where $\widetilde{Z} = \frac{Z}{Z_{\rm scale}}$, 
$\tilde{w} = Z_{\rm scale} w$, $\tilde{u} = Z_{\rm scale} u$, and
$\widetilde{B}=\{\tilde{w} \in \R^d \mid \norm{\tilde{w}}\leq Z_{\rm scale}\}$. Therefore, if we have computed an optimal solution $(r^*,\tilde{w}^*,\beta^*,\xi^*,\tilde{u}^*)$ of
\eqref{eq-gen-scale}, then
$(r^*,Z^{-1}_{\rm scale}\tilde{w}^*,\beta^*,\xi^*,Z^{-1}_{\rm scale}\tilde{u}^*)$
would be an optimal solution of \eqref{eq-gen}.

\subsection{Stopping condition for inexact sGS-ADMM}\label{subsec-stopping}

We measure the accuracy of an approximate optimal solution $(r,w,\beta,\xi,u,\alpha,\rho)$ for \eqref{eq-gen-scale} based on the KKT optimality conditions \eqref{eq-optimality} by defining the following relative residuals:
\begin{eqnarray*}
	\arraycolsep=1.4pt\def\arraystretch{1.5}
	\begin{array}{lllll}
		\eta_{C_1} = \frac{|y^T \alp|}{1+C}, & \eta_{C_2} = \frac{|\xi^T (Ce-\alp)|}{1+C}, & \eta_{C_3} = \frac{\norm{\alp - s}^2}{1+C} \text{ with } s_i = \frac{q}{r_i^{q+1}} , \\
		\eta_{P_1} = \frac{\norm{\widetilde{Z}^T \tilde{w} + \beta y + \xi - r}}{1+C}, 
		& \eta_{P_2} = \frac{\norm{D(\tilde{w}-\tilde{u})}}{1+C}, & \eta_{P_3} = \frac{\max\{\norm{\tilde{w}}-Z_{\rm scale},0\}}{1+C} , 
		\\
		\eta_{D_1} = \frac{\norm{\min\{0,\alp\}}}{1+C}, & \eta_{D_2} = \frac{\norm{\max\{0,\alp-Ce\}}}{1+C} ,
	\end{array}
\end{eqnarray*}
where $Z_{\rm scale}$ is a scaling factor which has been discussed in the last subsection.
Additionally, we calculate the relative duality gap by:
$$
\eta_{gap} := \frac{|\primobj-\dualobj|}{1+|\primobj|+|\dualobj|},
$$
where
$\primobj = \sum_{i=1}^n \frac{1}{r_i^q}+ C \inprod{e}{\xi}, \; 
\dualobj =  \kappa\sum_{i=1}^n \alp_i^{\frac{q}{q+1}}- Z_{\rm scale}\norm{\widetilde{Z}\alp}, \text{ with } 
\kappa = \frac{q+1}{q} q^{\frac{1}{q+1}}.$
We should emphasize that although for machine learning problems, a high accuracy solution is
usually not required,  it is important however to use
the KKT optimality conditions as the stopping criterion to find a moderately
accurate solution in order to design a robust solver.

We terminate the solver when $\max\{\eta_P,\eta_D\}<10^{-5}$,  $\min\{\eta_C,\eta_{gap}\}<\sqrt{10^{-5}}$, and $\max\{\eta_C,\eta_{gap}\} < 0.05$.
Here, $\eta_C = \max\{\eta_{C_1},\eta_{C_2},\eta_{C_3}\}, \; \eta_P = \max\{\eta_{P_1},\eta_{P_2},\eta_{P_3}\}$, and $ \eta_D = \max\{\eta_{D_1},\eta_{D_2}\}$.
Furthermore, the maximum number of iterations is set to be 2000.

\subsection{Adjustment of Lagrangian parameter $\sigma$ }

Based upon some preliminary experiments, we set our initial Lagrangian parameter $\sig$ to be $\sigma_0=\min\{10C, n\}^q$, where $q$ is the exponent in \eqref{eq-gen},
and adapt the following strategy to update $\sigma$ to improve the convergence speed of the algorithm in practice:

\begin{description}
	\item[Step 1.] Set $\chi=\frac{\eta_P}{\eta_D}$, where $\eta_P$ and $\eta_D$ are defined in subsection \ref{subsec-stopping};
	\item[Step 2.] If $\chi>\theta$, set $\sig_{k+1} = \zeta \sig_k$; elseif $\frac{1}{\chi}>\theta$, set $\sig_{k+1} = \frac{1}{\zeta} \sig_k$.
\end{description}
Here we empirically set $\theta$ to be 5 and $\zeta$ to be 1.1. Nevertheless, if we have either $\eta_P\ll\eta_D$ or $\eta_D\ll\eta_P$, then we would increase $\zeta$ accordingly, say 2.2 if $\max\{\chi,\frac{1}{\chi}\}>500$ or 1.65 if $\max\{\chi,\frac{1}{\chi}\}>50$.

\subsection{Performance of the sGS-ADMM on UCI data sets}
In this subsection, we test our algorithm on instances from the UCI data repository  (\citealt{UCI}). 
The datasets we have chosen here are all classification problems with two classes.
However, the size for each class may not be balanced. To tackle the case of uneven class proportions, we use the weighted DWD model discussed in \cite{qiao}. Specifically,
we consider the  model \eqref{eq-primal} using $e = \bfone$ and the term
$\sum_{i=1}^n 1/r_i^{q}$ is replaced by $\sum_{i=1}^n \tau_i^q/ r_i^{q}$,  with
the weights $\tau_i$  given as follows:
\begin{eqnarray*}
	\tau_i = \left\{ \begin{array}{ll}
		\frac{\tau_-}{\max\{\tau_+,\tau_-\}} & \mbox{if $y_i=+1$} \\[5pt]
		\frac{\tau_+}{\max\{\tau_+,\tau_-\}} & \mbox{if $y_i=-1$}
	\end{array}, \right.
\end{eqnarray*}
where $\tau_\pm =
\big(|n_\pm| K^{-1}\big)^{\frac{1}{1+q}}$.
Here $n_{\pm}$ is the number of data points with class label $\pm 1$ respectively and $K:={n}/{\log(n)}$ is a normalizing factor.

\noindent The performance of our inexact sGS-ADMM method on the UCI data sets:

We begin with the results of the Matlab version as a benchmark and compare it to the Vanilla python implementation to get a sense of the difference in performance. Then we present the optimised variation of the Vanilla Python implementation where we use the Numba JIT(Just in time compiler) to speed up any NumPy and SciPy calls(especially since we used quite a few of them to solve many matrix sub-problems). Following this, we present the performance of a parallel compute version of the Numba implementation. More details about the thought process behind using the following implementations and the improvements and limitations of using them will be presented in the next section(Section \ref{Trials}). 

Matlab Version : 
\begin{center}
	\spacingset{1.2}
	\begin{footnotesize}	

		\begin{longtable}{| c | c | c | c | c | c| c| c| c}
			\hline
			\multicolumn{1}{|c}{Data} & \multicolumn{1}{|c|}{$n$}
			& \multicolumn{1}{|c|}{$d$}
			& \multicolumn{1}{|c|}{$C$}
			& \multicolumn{1}{|c|}{Iter}  & \multicolumn{1}{|c|}{Time (s)}
			& \multicolumn{1}{|c|}{psqmr$|$double}
			& \multicolumn{1}{|c|}{Train-error (\%)}
			\\ \hline			
			\endhead
			 a8a &22696 & 123 &6.27e+02 &201  &0.62  &   0$|$201 &15.10
 \\[3pt]\hline
 a9a &32561 & 123 &6.49e+02 &201  &0.62  &   0$|$201 &14.93
 \\[3pt]\hline
 ijcnn1 &35000 &  22 &4.23e+03 &401  &1.13  & 0$|$401 &7.77
 \\[3pt]\hline
 rcv1 &20242 &44505 &9.18e+03 & 81  &3.10  & 234$|$49 &0.63
 \\[3pt]\hline
 w7a &24692 & 300 &5.95e+02 &701  &1.71  &   0$|$701 &1.17

            \label{table-UCI_matlab}
            \\ \hline

		\end{longtable}
	\end{footnotesize}		
\end{center}

Vanilla Python Version : 
\begin{center}
	\spacingset{1.2}
	\begin{footnotesize}	

		\begin{longtable}{| c | c | c | c | c | c| c| c| c}
			\hline
			\multicolumn{1}{|c}{Data} & \multicolumn{1}{|c|}{$n$}
			& \multicolumn{1}{|c|}{$d$}
			& \multicolumn{1}{|c|}{$C$}
			& \multicolumn{1}{|c|}{Iter}  & \multicolumn{1}{|c|}{Time (s)}
			& \multicolumn{1}{|c|}{psqmr$|$double}
			& \multicolumn{1}{|c|}{Train-error (\%)}
			\\ \hline
			\endhead
			 a8a &22696 & 123 &6.27e+02 &202  &7.12  &   0$|$202 &15.097521
 \\[3pt]\hline
 a9a &32561 & 123 &6.48e+02 &202  &7.54 &   0$|$202 & 14.904421
 \\[3pt]\hline
 ijcnn1 &35000 &  22 &4.22e+03 &401  &15.73  &   0$|$401 &7.800400
 \\[3pt]\hline
 rcv1 &20242 &44505 &9.18e+03 & 81  &46.57  & 234$|$49 &0.733333
 \\[3pt]\hline
 w7a &24692 & 300 &5.94e+02 &699  &33.47  &   0$|$699 &1.254312

			\label{tableUCI_pyVanilla}
            \\ \hline
		\end{longtable}
	\end{footnotesize}		
\end{center}

Python Numba CPU Version : 
\begin{center}
	\spacingset{1.2}
	\begin{footnotesize}	

		\begin{longtable}{| c | c | c | c | c | c| c| c| c}
			\hline
			\multicolumn{1}{|c}{Data} & \multicolumn{1}{|c|}{$n$}
			& \multicolumn{1}{|c|}{$d$}
			& \multicolumn{1}{|c|}{$C$}
			& \multicolumn{1}{|c|}{Iter}  & \multicolumn{1}{|c|}{Time (s)}
			& \multicolumn{1}{|c|}{psqmr$|$double}
			& \multicolumn{1}{|c|}{Train-error (\%)}
			\\ \hline
			\endhead
			 a8a &22696 & 123 &6.27e+02 &202  &0.56  &   0$|$202 &15.097521
 \\[3pt]\hline
 a9a &32561 & 123 &6.48e+02 &202  &0.59  &   0$|$202 & 14.904421
 \\[3pt]\hline
 ijcnn1 &35000 &  22 &4.22e+03 &401  &0.98  &   0$|$401 &7.800400
 \\[3pt]\hline
 rcv1 &20242 &44505 &9.18e+03 & 81  &2.42  & 234$|$ 49 &0.733333
 \\[3pt]\hline
 w7a &24692 & 300 &5.94e+02 &699  &1.67  &   0$|$699 &1.254312

			\label{table-UCI_pyNumba}
            \\ \hline
		\end{longtable}
	\end{footnotesize}		
\end{center}

\pagebreak

Python Numba Parallel Version : 
\begin{center}
	\spacingset{1.2}
	\begin{footnotesize}	

		\begin{longtable}{| c | c | c | c | c | c| c| c| c}
			\hline
			\multicolumn{1}{|c}{Data} & \multicolumn{1}{|c|}{$n$}
			& \multicolumn{1}{|c|}{$d$}
			& \multicolumn{1}{|c|}{$C$}
			& \multicolumn{1}{|c|}{Iter}  & \multicolumn{1}{|c|}{Time (s)}
			& \multicolumn{1}{|c|}{psqmr$|$double}
			& \multicolumn{1}{|c|}{Train-error (\%)}
			\\ \hline
			\endhead
			 a8a &22696 & 123 &6.27e+02 &202  &0.43  &   0$|$202 &15.097521
 \\[3pt]\hline
 a9a &32561 & 123 &6.48e+02 &202  &0.44  &   0$|$202 & 14.904421
 \\[3pt]\hline
 ijcnn1 &35000 &  22 &4.22e+03 &407  &1.12  &   0$|$407 &7.513421
 \\[3pt]\hline
 rcv1 &20242 &44505 &9.18e+03 & 81  &2.32  & 234$|$ 49 &0.733333
 \\[3pt]\hline
 w7a &24692 & 300 &5.94e+02 &699  &1.23  &   0$|$699 &1.254312

			\label{tableUCI_pyNumba_parallel}
            \\ \hline
		\end{longtable}
	\end{footnotesize}		
\end{center}

Table \ref{tableUCI_pyNumba_parallel} presents the number of iterations and runtime required, as well as training error produced when we perform our inexact sGS-ADMM algorithm to solve {16} data sets. Here, the running time is the total time spent in reading the training data and in solving
the DWD model. The timing
for getting the best penalty parameter C is excluded.
The results are generated using the exponent $q=1$.
In the table, ``psqmr" is the iteration count for the preconditioned symmetric quasi-minimal residual method for solving the linear system \eqref{eq-linsys}. A `0' for ``psqmr" means that we are using a direct solver as mentioned in subsection \ref{subsec-directSolver} and \ref{subsec-smw}. {Under the column ``double" in Table \ref{tableUCI_pyNumba_parallel}, we also record the number of iterations  for which the extra Step 1c is executed to ensure the convergence of Algorithm 1.

Denote the index set $S=\{i \mid y_i [\text{sgn}(\beta+x_i^T w)]\le 0, i=1,\ldots,n\}$ for which the data instances are categorized wrongly, where $sgn(x)$ is the sign function. The training and testing errors are both defined by $\frac{|S|}{n}\times 100 \%$, where $|S|$ is the cardinality of the set $S$.}

Our algorithm is capable of solving all the data sets, even when the size of the data matrix is huge.
In addition, for data with an unbalanced class size, such as w7a, our algorithm is able to produce a classifier with small training error.


\pagebreak
\section{Insights from translating the algorithm to code}\label{Trials}

My initial language of choice to attempt to convert the algorithm to code was python as to me it was the fastest way to get from pseudo code to production. However, the process was not as easy as it seemed. After writing in base python for a period of time, I ported the writing over to use NumPy to handle matrix calculations instead. However, there was a significant slowdown as there are no direct methods in NumPy to handle sparse matrices. Considering we are using high dimensional data and the whole problem is essentially a task to manipulate and obtain insights from this HDLSS, it does not make much implementational sense to continue using python's NumPy as that means that the task will not be efficiently handled. Hence, it was a fairly natural move to port over to use SciPy given that it extends NumPy methods (all of the Numpy functions are subsumed into the SciPy namespace) and it also has the functionality to handle sparse matrices (scipy.sparse). 

The next concern was then in which how would we like to consider these sparse vectors. 
i.e: 
\begin{lstlisting}[language=Python]
#CSR (Compressed Sparse Row)
scipy.sparse.csr_matrix()
'''
    CSR (Compressed Sparse Row): similar to COO, but compresses the 
    row indices. "Row Major order".
'''
#CSC (Compressed Sparse Column)
scipy.sparse.csc_matrix()
'''
    CSC (Compressed Sparse Column): similar to CSR except that values are 
    read first in the column direction. "Column major order". 
'''
\end{lstlisting}

During my early tests, I was blindly using the csr\_matrix and had multiple \emph{'dimension mismatch'} errors during unit testing. These were primarily due to using the $* operator$ to handle multiplication functions but due to the way the sparse matrix is packed in CSR we would have to use the multiply() function provided in SciPy to properly handle pointwise multiplication. 

Later I learned the csc\_matrix is more efficient at accessing column vectors and hence superior to the csr\_matrix when it came to column operations making it a natural choice over the csr\_matrix for the sGS-ADMM algorithm implementation as the majority of functions require a great deal of column-wise operations.

For readers of this paper here is a breakdown of some common ways to store sparse matrices:

\begin{itemize}

    \item COO (Coordinate list): stores a list of --- (row, column, value) tuples
    
    \item DOK (Dictionary Of Keys): a dictionary that maps \\
    \begin{scriptsize}
        \ ${(row_{element}, column_{element}): value_{element}}$\
    \end{scriptsize}
    \\Efficient classic hash table approach to set elements.

    \item List of Lists(LIL): LIL stores one list per row. The lil\_matrix format is row-based (so conversion to CSR is more efficient (as opposed to CSC) when one wants to do operations on it)

\end{itemize}

There are many more methods and one could possibly implement an even more efficient method by using one of the sparse matrix representation methods provided over at https://docs.scipy.org/doc/scipy/reference/sparse.html

Another find was using the os.path.join() function instead of explicitly writing the seperators when specifying files in the working directories to reduce cross-operating system errors when running the python code on it due to the Windows and Unix differing choice of slash for the file directory separator. 

From the results, we can observe the vanilla python implementation was painfully slow so we ran a test using precompiled numba and later ported over to numba\_SciPy since the sparse matrix calculations were handled by SciPy instead. The results proved to be way better with significant improvements in the threaded variants where 
\begin{lstlisting}[language=Python]
@numba.vectorize(["float64(float64,float64)"],
    nopython=True,target='cpu')
\end{lstlisting}
was replaced with
\begin{lstlisting}[language=Python]
@numba.vectorize(["float64(float64,float64)"],
    nopython=True,target='parallel')
\end{lstlisting}

However, this improvement did not come without its faults. One a few runs on the test virtual machine I realised there was frequent crashing when using the Numba variant and later realised this was due to RAM limitations. Early tests on 32bit variants of Python 3.10 resulted in a memory error. This is because 32bit Python has theoretical access to approximately 4GB of RAM. In reality, this 4GB limit is actually way lower and closer to around 2GB due to 32bit operating system overheads. The first workaround is to install a 64bit version of Python on a 64bit operating system and this increases the ram limit. Alternatively, if we are limited to a 32bit host we can use numpy.memmap() function to map the array to disk and refactor our matrix handling code accordingly. We should note that the latter method is obviously going to have greater execution time overhead as increased I/O (input/output) calls due to the need to load/dump (read/write) to disk constantly to run within given memory limitations. Considering our HDLSS data sets, this is a serious issue when it comes to running on systems of low specifications or in general on systems with limited resources as unexpected failure due to memory limit exceptions is a real concern when it comes to real-world deployment.

Hence, despite the marginal speed improvement over Matlab when using the numba variant, given the current state of architecture, I would still highly recommend using Matlab for this task as you would have a better memory-usage-to-speed ratio and reliability. We have not tested this implementation using the Matlab engine inside of python or vice-versa but that is a possible future trial we could attempt. Another thing which we are interested to look into is building a version for Julia as it is known to be faster than Python and Matlab (\citealt{danielsson_lin_2022}) in the scientific computing domain given a task that requires large data sets. The next section will focus on more future work that could be done to push this research further in terms of speed and improving space consumption with a greater algorithmic focus rather than on the technological stack so as to generalise it for future technologies.

\pagebreak
\section{Future Expansions and Considerations}\label{Future}
\subsection{Psuedocode}
\noindent \textbf{Let us begin by looking at the pseudo-code of our DWD solver}:

\begin{algorithm}
	\caption{Main $f$unction} 
	\begin{algorithmic}[1]
		\For {LIBSVM data $= a,b,c,\ldots$}
				\State Convert LIBSVM data to Sparse Matrix for computation
				\State Remove zero features from Sparse Matrix
                \State Scale the resultant matrix's features to roughly have the same magnitude
                \State Compute penalty parameter
                \State Run sGS-ADMM on refactored data using parameters obtained above
		\EndFor
	\State Print Information

	\end{algorithmic} 
\end{algorithm}

\begin{algorithm}
	\caption{sGS-ADMM} 
	\begin{algorithmic}[1]
        
        \State Set initial iterates
        \If {$dim > 5000$ and $ n < 0.2*dim $ and $ n <= 2500)$} 
            \State $Solver$ = 'SMW2'
             \ElsIf {$dim > 5000$}
             \State $Solver$ = 'iterative'
             \Else
             \State $Solver$ = 'direct'
        \EndIf
        \State Apply Cholesky Decomposition based on $Solver$ \hfill //(Refer to section 3.3)
		\For {$iteration=1,2,\ldots , maxIterate$}
                        \Switch{$Solver$} \hfill //(update $\omega$,$\beta$)
                            \Case{$smw2$}
                              \State update $\omega$,$\beta$ using \emph{smw2}
                            \EndCase
                            \Case{$direct$}
                              \State update $\omega$,$\beta$ using \emph{linSysSol} \hfill //\begin{footnotesize}(linSysSol shown algo 3) \end{footnotesize}
                            \EndCase
                            \Case{$iterative$}
                              \State update $\omega$,$\beta$ using \emph{psqmr}
                            \EndCase
                        \EndSwitch

                    \State Update r \hfill //\begin{footnotesize}
                    (sub-step of which is \emph{Newton-Raphson} Root finding algorithm)
                    \end{footnotesize}
				\State Update  $\omega$,$\beta$ again  OR  directly extended ADMM //\begin{footnotesize}(depends on provided method) \end{footnotesize}
			\EndFor
			\State Check for termination
			\State Adjust $\sigma$ 

	\end{algorithmic} 
\end{algorithm}

\begin{algorithm}
	\caption{linSysSol [linearSystemSolver]} 
	\begin{algorithmic}[1]
        \State Given [L.R,indef,L.perm] = choleskyFactorisation(matrix)
        \If {L.perm exists} 
            \If {$Full Cholesky factorisation$} 
                \State $q(L.perm,1) = mextriang(L.R, mextriang(L.R,r(L.perm),2) ,1)$
                
                \ElsIf {$Sparse Cholesky factorisation$} 
                
                    \State $q(L.perm,1) = mexbwsolve(L.Rt,mexfwsolve(L.R,r(L.perm)))$
            \EndIf
            
        \Else 
            \If {$Full Cholesky factorisation$} 
                \State $q = mextriang(L.R, mextriang(L.R,r,2) ,1)$
                
                \ElsIf {$Sparse Cholesky factorisation$} 
                
                    \State $q = mexbwsolve(L.Rt,mexfwsolve(L.R,r))$
            \EndIf

        \EndIf
    
	\end{algorithmic} 
\end{algorithm}

{\raggedleft\vfill{%
    \begin{FlushRight}
        \begin{spacing}{0.25}
            \begin{tiny}
                \emph{ (note: much of the given code has been abstracted with some of the internals outright skipped to greater focus on specific parts which we would like to discuss later)}
                \end{tiny}
            \end{spacing}
        \end{FlushRight}
}\par
}
\pagebreak
\emph{\textbf{Where}}  Given upper-triangular matrix X
\begin{itemize}
    \item mextriang(X,b,options) solves $X *y = b$ given option 1 and $X'*y = b$ given option 2 
    \item mexbwsolve(transpose(X),b) solves $X*y = b$ 
    \item mexfwsolve(X,b) solves $X'*y = b$ 
\end{itemize}

\subsection{Alternate Algorithm Suggestions}
Note the following section focuses on methods to improving "\textbf{\emph{Time complexity}}" or "\textbf{\emph{Space complexity}}" or "\textbf{\emph{Memory Utilisation}}":

We could consider replacing the current Cholesky/LU/QR decomposition implementations with the one that uses Fast Rectangular Matrix Multiplication (\citealt{camerero_2018}) to push the time complexity of those operations down to \textbf{O($n^{2.529}$)}.

Alternatively, we could consider implementing Cholesky factorisation by Kullback-Liebler divergence minimization (\citealt{Schafer5_2020}) which would yield a time complexity improvement of \textbf{O($N \log(N/\epsilon)^2d$)} and a space complexity improvement of \textbf{O($N \log(N/\epsilon)^d$)}.

We could also look into implementing Broyden's method (a  quasi-Newton method) instead of Newton-Raphson for root finding task to improve memory utilisation by the algorithm as Broyden is superior when it comes to storage and approximation of the Jacobian (\citealt{ramli_abdullah_mamat_2010}).

Considering that quasi-Newton methods (\citealt{cericola_2015}) approximate the inverse Hessian Matrix, and hence, unlike full Newton-Raphson, avoid iteratively calculating the inverse Hessian:
The lack of a second derivative and requirement to solve a linear system of equations hence results in a computationally cheaper method present for us to use.
However, more convergence steps and  lack of precision in the Hessian calculation lead to slower convergence in terms of steps and hence, a less precise convergence path. In the case of simpler root-finding problems where the extra computation time to actually compute the Hessian inverse is low. Another possible issue is that to store the inverse Hessian approximation, a large amount of memory may be required for our high dimensional problems. Thus, in the vein of quasi-Newton methods, we have to carefully consider which methods would fit best for our use case.
\pagebreak
\subsection{Parallelism}
Note the following section focuses on methods to decrement "\textbf{\emph{Execution time}}" rather than improving "\textbf{\emph{Time complexity}}":
\\
\noindent Ideally, to speed up the execution time of the above algorithm we can look at parallelism methods to use more cores and threads to push for faster I/O (input/output). Considering the simple approach of increasing thread count, we can achieve a multiplicative effect on decreasing execution time as multi-threading results in greater optimisation of the CPU usage and hence a performance increase. However,  although increasing the thread count sounds like a great idea, all good things come with a price. We are, of course, in this field of optimisation cause we want to find optimality in these cost-benefit ratios. This is why I thought this would be an interesting problem to tackle. \\

\noindent Let us begin by contextualising our playing field through first understanding some underlying risks of using threading methods:

\begin{enumerate}
\item Context Switching Overhead

Although each thread's execution time decreases as the number of threads increases, there is an increase in the overhead cost caused by switching the application context between threads, and potential pre-processing and post-processing calls because of this switch. This implies resulting degradation of performance gains when there are more than the optimal number of threads.

\item  Increased Resource Consumption

Threads require some memory to operate and too many threads can bog down memory. This can cause unnecessary strain imposed on a poorly optimised scheduler which may malfunction and possibly crash(on a Linux machine one can use clone(2)(\citealt{eckhard_1992}) or similar processes to manage memory usage from thread creation to attempt to mitigate memory strain and overload from creating too many threads)
\pagebreak

\item Thread Safety

Consideration of design patterns and a complete understanding of the program flow is key to successful multithreading. The price of failure to understand the two, maybe in the case where instead of threading mutually independent tasks one tries to thread mutually dependent ones, would lead to catastrophic failure. In poorly multithreaded programs, one can face complications such as Forgotten Synchronization, Incorrect Granularity, Read and Write Tearing, Lock-Free Reordering, Lock Convoys, Two-Step Dance and Priority Inversion(\citealt{duffy_2008}). Achieving thread safety is possible via the use of immutable implementations,thread-local via private fields and local variable scoping, stateless deterministic functions, synchronization wrappers, atomic objects and functions, and locks.

\item  Complex design 

Considering all the above implications and design requirements, the maintainability and readability of source code would decay significantly due to the implementation of threading. Debugging the code base could become extremely hard and also increased risk of deadlocks due to poor implementation is of great concern (Thread Analysis tools like VizTracer for python could be used to possibly mitigate this but this leads to increased development time and increased maintenance costs)

\end{enumerate}

\noindent Considering our algorithm and the above guidelines on threadings we propose the following sub-problems that could potentially be threaded:
\begin{itemize}
  \item The first for loop in main could be threaded so each read is on a separate thread
  \item Choleksy Factorisation (using block algorithms (\citealt{Povelikin_2019}))
  \begin{itemize}
    \item There is the possibility of performantly implementing by converting the standard matrix multiplication approach in C (\citealt{boson_2014})
    \end{itemize}
   \item All the mex-matrix solvers (By extending the Algorithms-by-Blocks(\citealt{quintana-orti_geijn_zee_chan_2009}) approach)
\end{itemize}

Hence, finding an optimal thread count(\citealt{pavel_kazenin's_blog_2014}) for each of the thread-safe sub-problems would be an ideal future goal.

\spacingset{0.9}
\pagebreak
\bibliography{Bibliography-DWD}
\end{document}